\documentclass[11pt]{article}
\usepackage{amsfonts}
\usepackage{latexsym}
\usepackage{cite}
\usepackage{amsmath,amsfonts,latexsym,amssymb}
\usepackage[mathscr]{eucal}
\usepackage{cases}
\usepackage{amsthm}

\usepackage[bf,small]{caption2}
\usepackage{float}
\usepackage{graphicx}
\usepackage{amsmath}
\usepackage{amssymb}
\usepackage[all]{xy}

\newcommand\Prefix[3]{\vphantom{#3}#1#2#3}

\newtheorem{theorem}{theorem}[section]
\newtheorem{thm}[theorem]{Theorem}
\newtheorem{lem}[theorem]{Lemma}

\newtheorem{defn}[theorem]{Definition}
\newtheorem{exmp}[theorem]{Example}

\newtheorem{rmk}[theorem]{Remark}
\newtheorem{nota}[theorem]{Notation}

\begin{document}

\title{\textbf{The Dijkgraaf-Witten invariants of Seifert 3-manifolds with orientable bases}}
\author{\Large Haimiao Chen
\footnote{Email: \emph{chenhm@math.pku.edu.cn}. The author is supported by NSFC-11401014.} \\
\normalsize \em{Beijing Technology and Business University, Beijing, China}}
\date{}
\maketitle

\begin{abstract}
  We derive a formula for the Dijkgraaf-Witten invariants of Seifert 3-manifolds with orientable bases.
\end{abstract}

\section{Introduction}

Let $\Gamma$ be a finite group and let $\omega$ be a $U(1)$-valued 3-cocycle on $B\Gamma$.
For a closed oriented 3-manifold $M$, the \emph{Dijkgraaf-Witten invariant} (DW invariant for short) of $M$ is defined as
\begin{align}
Z^{\omega}(M)=\frac{1}{\#\Gamma}\cdot\sum\limits_{\Phi:\pi_{1}(M)\rightarrow\Gamma}\langle F_{\Phi}^{\ast}[\omega],[M]\rangle,  \label{eq:DW}
\end{align}
where $F_{\Phi}:M\rightarrow B\Gamma$ is a map inducing $\Phi$ on the fundamental group, which is determined by $\Phi$ up to homotopy, $[M]$ is the fundamental class of $M$,
and $\langle,\rangle$ is the pairing $H^{3}(M;U(1))\times H_{3}(M;\mathbb{Z})\rightarrow U(1)$.
Clearly $Z^{\omega}(M)$ depends only on the cohomology class $[\omega]$.

DW invariant is the partition function of a 3-dimensional topological quantum field theory, {\it Dijkgraaf-Witten theory}, which was first proposed
by the two authors \cite{DW} naming it and then further developed by Wakui \cite{Wa92}, Freed-Quinn \cite{FQ}, Freed \cite{Fr94}; see also \cite{ETQFT, Wi08} and the references therein. Freed \cite{Fr94} generalized DW theory to arbitrary dimension.

DW invariant encodes information on fundamental group and fundamental class, so it will not be surprising to see some connection between DW invariant and classical invariants. As an evidence, DW invariant was used by Chen\cite{Ch14} to count homotopy classes of maps from a closed manifold to a topological spherical space form with given degree.

Till now, there are few computations seen in the literature. In dimension 2, Turaev \cite{Tu07} derived the formula for the DW
invariants of closed surfaces using projective representations of finite groups, and Khoi \cite{Kh11} extended the result to
surfaces with boundary. In dimension 3,
Chen \cite{Ch12} gave a formula for the DW invariants of Seifert 3-manifolds when $\omega$ is trivial; Khoi \cite{Kh14} computed the general DW invariants for circle boundles; Matveev and Turaev \cite{MT15} expressed the DW invariant over $\mathbb{Z}_{2}$ via Arf invariant.
As another contribution, in this article we compute the general DW invariants of Seifert 3-manifolds with orientable bases.
%Finally we shall point out that, Hansen \cite{RT} had derived a formula for the Reshetikhin-Turaev invariants of Seifert 3-manifolds,
%associated to any modular tensor category. In fact, as proved by Freed \cite{Fr99}, DW theory is a TQFT of Reshetikhin-Turaev type, the
%corresponding modular tensor category being the representation category of the twisted quantum double. It could be said that our result is essentially not new. However, our approach unfolds many details and highlights the essence, and leads to a formula which is
%more explicit and elegant.

The content is organized as follows.
In Section 2 we give an exposition of DW theory, reviewing the construction of \cite{Fr94} in dimension 3. We adopt algebriac and technical notions, instead of geometric and conceptual ones used in \cite{Fr94}; the aim is to make computations easier to get started, and to lay the foundation for deriving more formulas for DW invariants.
Section 3 is devoted to computing the DW invariants of Seifert 3-manifolds with orientable bases;
the result is expressed in terms of irreducible characters of the ``twisted quantum double" of $\Gamma$, as called in the literature.
We go along the same line as in \cite{Ch12}. However, it turns out that some finer structures should be taken care of, and actually this constitutes the main difficulty.

\section{Exposition of Dijkgraaf-Witten theory}

\subsection{Preparation}

The 3-cocycle $\omega$ can be identified with a function $\Gamma^{3}\to U(1)$ such that
\begin{align*}
\omega(y,z,w)\omega(xy,z,w)^{-1}\omega(x,yz,w)\omega(x,y,zw)^{-1}\omega(x,y,z)=1
\end{align*}
for all $x,y,z,w\in\Gamma$.
We assume that it is normalized, i.e., $\omega(x,y,z)=1$ if $e\in\{x, y, z\}$.
Recall that (see \cite{AT} Page 89) $B\Gamma$ has a simplicial model, in which a $k$-simplex is an ordered $k$-tuple $[x_{1}|\cdots|x_{k}]$ with $x_{i}\in\Gamma$, and
\begin{align*}
\partial[x_{1}|\cdots|x_{k}]=[x_{2}|\cdots|x_{k}]+\sum\limits_{i=1}^{k-1}(-1)^{i}[x_{1}|\cdots|x_{i}x_{i+1}|\cdots x_{k}]+ [x_{1}|\cdots|x_{k-1}].
\end{align*}
The value of $\omega$ taking at the 3-simplex $[x|y|z]$ is $\omega(x,y,z)$.

All manifolds are assumed to be compact and oriented. For a manifold $M$, let $-M$ denote the manifold obtained by reversing the
orientation of $M$. All diffeomorphisms are assumed to preserve orientation.

For a topological space $X$, let $\Pi_{1}(X)$ denote the fundamental groupoid of $X$; if $\gamma$ is a path in $X$, then denote its homotopy class
(fixing the endpoints) also by $\gamma$.
Let $\mathfrak{B}(X)=\mathcal{F}(\Pi_{1}(X),\Gamma)$, the functor category, (viewing $\Gamma$ as a groupoid with a single object); note that a morphism $\lambda:\phi\rightarrow\phi'$ is a map $\lambda:X\rightarrow\Gamma$ such that
$$\phi'(\gamma)=\lambda(\gamma(0))^{-1}\phi(\gamma)\lambda(\gamma(1))$$
for each path $\gamma:[0,1]\rightarrow X$.
If $\pi_{1}(X)$ is finitely generated, then $\mathfrak{B}(X)$ has finitely many connected components.

Let $\Delta^{k}=[v_{0},\ldots,v_{k}]$ denote the standard $k$-simplex. For each singular $k$-simplex
$\sigma:\Delta^{k}\to X$, let $\sigma_{i}=\sigma(v_{i})$ and let $\sigma_{ij}=\sigma|_{[v_{i},v_{j}]}:[v_{i},v_{j}]\to X.$

Let $\lambda:\phi\rightarrow\phi'$ be a morphism in $\mathfrak{B}(X)$. Given $\sigma:\Delta^{k}\to X$, put
\begin{align}
\sigma\langle\phi\rangle&=[\phi(\sigma_{01})|\cdots|\phi(\sigma_{k-1,k})]\in C_{k}(B\Gamma;\mathbb{Z}), \\
\sigma\langle\lambda\rangle&=\sum\limits_{j=0}^{k} (-1)^{k-j}
[\phi(\sigma_{01})|\cdots|\phi(\sigma_{j-1,j})|\lambda(\sigma_{j})|\phi'(\sigma_{j,j+1})|\cdots|\phi'(\sigma_{k-1,k})] \nonumber \\
&\in C_{k+1}(B\Gamma;\mathbb{Z});
\end{align}
given $\xi=\sum\limits_{i}n_{i}\sigma^{i}\in C_{k}(X;\mathbb{Z})$, put
\begin{align}
\xi\langle\phi\rangle&=\sum\limits_{i}n_{i}\cdot\sigma^{i}\langle\phi\rangle\in C_{k}(B\Gamma;\mathbb{Z}), \\
\xi\langle\lambda\rangle&=\sum\limits_{i}n_{i}\cdot\sigma^{i}\langle\lambda\rangle\in C_{k+1}(B\Gamma;\mathbb{Z}).
\end{align}

\subsection{The DW invariant of a closed surface}

For the ``empty surface" $\emptyset$, set $Z^{\omega}(\emptyset)=\mathbb{C}$, equipped with the standard inner product.

Suppose $N$ is a nonempty closed surface.
Given $\phi\in\mathfrak{B}(N)$ and $\tau\in C_{2}(N;\mathbb{Z})$ homologous to zero, choose
$\varsigma\in C_{3}(N;\mathbb{Z})$ with $\partial\varsigma=\tau$, and set
\begin{align}
\omega(\phi;\tau)=\omega(\varsigma\langle\phi\rangle).  \label{eq:df}
\end{align}
This is independent of the choice of $\varsigma$: if also $\partial\varsigma'=\tau$, then $\varsigma'-\varsigma=\partial\nu$ for some $\nu\in C_{4}(N;\mathbb{Z})$, for $H_{3}(N;\mathbb{Z})=0$, hence
$$\frac{\omega(\varsigma'\langle\phi\rangle)}{\omega(\varsigma\langle\phi\rangle)}=
\omega(\partial\nu\langle\phi\rangle)=(\delta\omega)(\nu\langle\phi\rangle)=1.$$

When $N$ is connected, the fundamental class $[N]$ is the set of singular 2-cycles representing $[N]$. When $N$ has connected components $N_{1},\ldots,N_{r}$, let $[N]=[N_{1}]\times\cdots\times[N_{r}]$; note that
$\mathfrak{B}(N)\cong\mathfrak{B}(N_{1})\times\cdots\times\mathfrak{B}(N_{r})$.

\begin{defn}
\rm The DW invariant $Z^{\omega}(N)$ is the Hermitian space, whose underlying vector space consists of functions
$\vartheta:[N]\times\mathfrak{B}(N)\to\mathbb{C}$
such that
\begin{align}
\vartheta(\xi',\phi')=\omega(\xi'\langle\lambda\rangle)\omega(\phi;\xi'-\xi)\cdot\vartheta(\xi,\phi) \label{eq:condition}
\end{align}
for any $\xi,\xi'$ and any $\lambda:\phi\rightarrow\phi'$,
and whose inner product is given by
\begin{align}
(\vartheta_{1},\vartheta_{2})=\sum\limits_{[\phi]\in\pi_{0}(\mathfrak{B}(N))}
\frac{1}{\#\textrm{Aut}(\phi)}\cdot\vartheta_{1}(\xi,\phi)\overline{\vartheta_{2}(\xi,\phi)},  \label{eq:inner}
\end{align}
where $\xi\in[N]$ and $\phi\in[\phi]$ are taken arbitrarily.
\end{defn}

\begin{rmk} \label{rmk:key}
\rm Suppose $S$ is a subset of $N$ which intersects each connected component of $N$.
Let $\Pi_{1,S}(N)$ be the full subcategory of $\Pi_{1}(N)$
with object set $S$, let $[N]_{S}$ be the subset of $[N]$ consists of $\sum\limits_{i}n_{i}\sigma^{i}$ with $\sigma^{i}_{0}, \sigma^{i}_{1}, \sigma^{i}_{2}\in S$, and let
$$\mathfrak{B}_{S}(N)=\mathcal{F}(\Pi_{1,S}(N),\Gamma).$$
Define $Z^{\omega}_{S}(N)$ similarly as above, with $[N]$, $\mathfrak{B}(N)$ replaced by $[N]_{S}$, $\mathfrak{B}_{S}(N)$, respectively. Then there is a canonical isometry
$$Z^{\omega}_{S}(N)\to Z^{\omega}(N), \ \ \ \ \ \varphi\mapsto\check{\varphi},$$
with $\check{\varphi}$ determined by $\check{\varphi}(\xi,\phi)=\varphi(\xi,\phi_{S})$, for any $\xi\in[N]_{S}$ and any $\phi\in\mathfrak{B}(N)$, where $\phi_{S}$ is the pull-back under $\Pi_{1,S}(N)\to\Pi_{1}(N)$.

For practical computation, it is convenient to choose an appropriate $S$ (usually a finite set) and identify $Z^{\omega}(N)$ with $Z^{\omega}_{S}(N)$.
\end{rmk}

\subsection{The DW invariant of a 3-manifold}

Let $M$ is a 3-manifold.
If $\partial M=\emptyset$, then $Z^\omega(M)\in\mathbb{C}$ has been defined as (\ref{eq:DW}).
Suppose $M=M_1\sqcup M_2$ with $\partial M_1=\emptyset$ and each connected component of $M_2$ has nonempty boundary, so that $\partial M=\partial M_2$.

Given $\xi\in[\partial M]$ and $\phi\in\mathfrak{B}(\partial M)$,
take $\varsigma\in [M_2,\partial M]$ with $\partial\varsigma=\xi$, and %If $\Phi\in\mathfrak{B}(M)$ with $\Phi|_{\partial M}=\phi$,
let $\mathfrak{B}_{\phi}(M_2)$ be the subcategory of $\mathfrak{B}(M_2)$ whose objects are the functors $\Phi:\Pi_{1}(M_2)\rightarrow\Gamma$
such that the composite $\Pi_{1}(\partial M)\to\Pi_{1}(M_2)\to\Gamma$ is equal to $\phi$, and whose morphisms are maps $\lambda:M_2\rightarrow\Gamma$ with $\lambda|_{\partial M}\equiv e$.
For $\Phi\in\mathfrak{B}_{\phi}(M_2)$, due to the reason similarly as in the previous subsection,
$\omega(\varsigma\langle\Phi\rangle)$ is independent of the choice of $\varsigma$.
Furthermore, if $\lambda:\Phi\rightarrow\Phi'$ is a morphism in $\mathfrak{B}_{\phi}(M_2)$, then
$$\frac{\omega(\varsigma\langle\Phi'\rangle)}{\omega(\varsigma\langle\Phi\rangle)}=
\omega(\partial\varsigma\langle\lambda\rangle)=(\delta\omega)(\varsigma\langle\lambda\rangle)=1.$$
Thus $\omega(\varsigma\langle\Phi\rangle)$ depends only on $\xi$ and the connected component containing $\Phi$;
denote it by $\omega(\xi,[\Phi])$.

\begin{defn}
\rm The DW invariant $Z^{\omega}(M)$ is the function
\begin{align}
Z^{\omega}(M)&: [\partial M]\times\mathfrak{B}(\partial M)\to\mathbb{C}, \\
(\xi,\phi)&\mapsto Z^\omega(M_1)\cdot\sum\limits_{[\Phi]\in\pi_{0}(\mathfrak{B}_{\phi}(M_2))}\omega(\xi,[\Phi]).
\end{align}
\end{defn}

\subsection{The TQFT axioms}

Assertion 4.12 of \cite{Fr94} can be restated as
\begin{thm}
To each closed surface $N$ is assigned a Hermitian space $Z^{\omega}(N)$,
and to each 3-manifold $M$ is assigned a vector $Z^{\omega}(M)\in Z^{\omega}(\partial M)$.
They satisfy the following:

{\rm(a)} (Functorality) Each diffeomorphism $f:N\rightarrow N'$ induces an isometry
$$f_{\ast}:Z^{\omega}(N)\to Z^{\omega}(N').$$
For each diffeomorphism $F:M\rightarrow M'$, one has
$(\partial F)_{\ast}(Z^{\omega}(M))=Z^{\omega}(M').$

{\rm(b)} (Orientation) There is a canonical isometry
$$Z^{\omega}(-N)\cong \overline{Z^{\omega}(N)},$$
through which $Z^{\omega}(-M)$ is sent to $Z^{\omega}(M)$.

Here $\overline{Z^{\omega}(N)}$ is obtained from re-equipping $Z^{\omega}(N)$ with the scalar product $(\lambda,v)\mapsto\overline{\lambda}v$ and the inner product
$(u,v)_{\overline{Z^{\omega}(N)}}=(v,u)_{Z^{\omega}(N)}.$

{\rm(c)} (Multiplicativity)  There is a canonical isometry
$$Z^{\omega}(N\sqcup N')\cong Z^{\omega}(N)\otimes Z^{\omega}(N'),$$
through which $Z^{\omega}(M\sqcup M')$ is sent to $Z^{\omega}(M)\otimes Z^{\omega}(M')$.

{\rm(d)} (Gluing) Let $M$ be a 3-manifold with $\partial M=-N\sqcup N'\sqcup N_0$, and let $\check{M}$ be the 3-manifold obtained from gluing $M$ via a diffeomorphism $f:N\rightarrow N'$, i.e., $\check{M}=M/x\sim f(x).$
Then $Z^{\omega}(M)$ is sent to $Z^{\omega}(\check{M})$ by the composite
\begin{align*}
Z^{\omega}(\partial M)\cong\overline{Z^{\omega}(N)}\otimes Z^{\omega}(N')\otimes Z^{\omega}(N_0)\rightarrow Z^{\omega}(N_0)=Z^{\omega}(\partial\check{M}),
\end{align*}
where the middle map uses the pairing
$$\overline{Z^{\omega}(N)}\otimes Z^{\omega}(N')\to\mathbb{C},  \ \ \
u\otimes v\mapsto(v,f_{\ast}(u))_{Z^{\omega}(N')}.$$
\end{thm}

\begin{rmk}
\rm It is worth describing the isometry $f_{\ast}$ in (a) explicitly. Let $f_{\#}:[N]\rightarrow [N']$ be the induced map on singular chains, and $f^{\ast}:\mathfrak{B}(N')\rightarrow\mathfrak{B}(N)$ the pullback.
For $\vartheta\in Z^{\omega}(N)$,  the function $f_{\ast}(\vartheta):[N']\times\mathfrak{B}(N')\to\mathbb{C}$ takes the form
\begin{align}
(\xi',\phi')\mapsto\vartheta(f_{\#}^{-1}(\xi'),f^{\ast}(\phi'))=
\omega(f^{\ast}(\phi');f_{\#}^{-1}(\xi')-\xi')\cdot\vartheta(\xi',f^{\ast}(\phi')), \label{eq:diff}
\end{align}
the equality following from (\ref{eq:condition}).
\end{rmk}

\begin{rmk}
\rm If $\partial M=-N\sqcup N'$, i.e., $M$ is a cobordism from $N$ to $N'$, then
$Z^{\omega}(M)$ can be identified with a linear map $Z^{\omega}(N)\to Z^{\omega}(N')$ through
$$Z^{\omega}(\partial M)\cong Z^{\omega}(-N)\otimes Z^{\omega}(N')\cong\overline{Z^{\omega}(N)}\otimes Z^{\omega}(N')\cong\hom(Z^{\omega}(N),Z^{\omega}(N')).$$
Explicitly, for each $\theta\in Z^\omega(N)$ and $(\xi',\phi')\in[N']\times\mathcal{B}(N')$,
$$Z^\omega(M)(\theta): (\xi',\phi')\mapsto\sum\limits_{[\phi]\in\pi_0(\mathfrak{B}(N))}
\frac{(\#\Gamma)^{\#\pi_0(N)}}{\#{\rm Aut}(\phi)}I(\theta;(\xi,\phi),(\xi',\phi')),$$
where $\xi$ is an arbitrarily chosen element in $[N]$, and
$$I(\theta;(\xi,\phi),(\xi',\phi'))=
\sum\limits_{[\Phi]\in\pi_0(\mathfrak{B}_{\phi\sqcup\phi'}(M))}
\theta(\xi,\phi)\omega(\xi'-\xi,[\Phi]).$$

According to the gluing axiom, if $M'$ is another cobordism from $N'$ to $N''$ and we glue $M$ with $M'$ via a diffeomorphism $f:N'\to N'$, then the linear map $Z^{\omega}(M\cup_{f}M'):Z^{\omega}(N)\to Z^{\omega}(N'')$ is equal to
$Z^{\omega}(M')\circ f_{\ast}\circ Z^{\omega}(M)$.
\end{rmk}

\section{Computing the DW invariants of Seifert 3-manifolds}

\subsection{The Hermitian vector space $Z^{\omega}(\Sigma_{1})$} \label{sec:E}

\begin{nota}
\rm For $h,h',x,x'\in\Gamma$, let
\begin{align}
\gamma^{\omega}_{x}(h,h')&=\omega(h,h',x)\omega(x,x^{-1}hx,x^{-1}h'x)\omega(h,x,x^{-1}h'x)^{-1},  \\
\theta^{\omega}_{x}(h,h')&=\omega(x,h,h')\omega(h,h',(hh')^{-1}xhh')\omega(h,h^{-1}xh,h')^{-1}.   \label{eq:theta}
\end{align}
\end{nota}

\begin{nota}
\rm Let $\pi:\mathbb{R}^{2}\rightarrow S^{1}\times S^{1}=\Sigma_{1}$ be the universal covering.
For $A_{i}\in\mathbb{Z}^{2}\subset\mathbb{R}^{2}, i=0,1,2$, use $[A_{0}A_{1}A_{2}]$ to denote the singular 2-simplex
$\Delta^{2}\stackrel{\tilde{\sigma}}\rightarrow\mathbb{R}^{2}\stackrel{\pi}\rightarrow\Sigma_{1},$
where $\tilde{\sigma}$ is the map extending $v_{i}\mapsto A_{i}$ linearly.

Similarly for singular 3-simplices.
\end{nota}

In the notation of Remark \ref{rmk:key}, let $S=\{1\times 1\}\subset\Sigma_{1}$.
%consisting of $\mathbb{Z}$-linear combinations of 2-simplicies of the form $[ABC]$ with $A,B,C\in\mathbb{Z}^{2}$.
Let
\begin{align}
\xi_{0}&=[OW_{1}W_{3}]-[OW_{2}W_{3}]\in [\Sigma_{1}]_{S}, \\
\text{with \ \ \ \ } O&=(0,0), \ \ \ W_{1}=(0,1), \ \ \ W_{2}=(1,0), \ \ \ W_{3}=(1,1).
\end{align}
An object of $\mathfrak{B}_{S}(\Sigma_{1})$ is a homomorphism
$\phi:\pi_{1}(\Sigma_{1},1\times 1)\rightarrow\Gamma$, which is determined by its values at ${\rm mer}=S^{1}\times 1$ and ${\rm lon}=1\times S^{1}$;
write $\phi$ as $\phi_{x,h}$ if $\phi({\rm mer})=x$ and $\phi({\rm lon})=h$.
A morphism $\phi_{x_{1},h_{1}}\to\phi_{x_{2},h_{2}}$ is an element $u$ such that $\Prefix^{u}{x_{1}}=x_{2}$ and $\Prefix^{u}{h_{1}}=h_{2}$.

The vector space $E=Z^{\omega}_{S}(\Sigma_{1})$ consists of functions
$$\varphi:\{(x,h)\in\Gamma^{2}\colon xh=hx\}\rightarrow\mathbb{C}$$
such that
\begin{align}
\varphi(u^{-1}xu,u^{-1}hu)=\frac{\gamma^{\omega}_{u}(x,h)}{\gamma^{\omega}_{u}(h,x)}\cdot\varphi(x,h) \ \ \ \ \
\text{for\ all\ \ } u\in\Gamma,
\end{align}
and is equipped with the inner product
\begin{align}
(\varphi,\varphi')_{E}=\frac{1}{\#\Gamma}\cdot\sum\limits_{x,h\colon xh=hx}\varphi(x,h)\overline{\varphi'(x,h)}.
\end{align}
By Remark \ref{rmk:key}, there exists a canonical isometry
\begin{align}
E\cong Z^{\omega}(\Sigma_{1}), \ \ \ \ \ \varphi\mapsto\vartheta \ \ \text{with\ \ }
\vartheta(\xi_{0},\phi)=\varphi(\phi({\rm mer}),\phi({\rm lon})).  \label{eq:iso}
\end{align}

Let $\mathcal{C}(\Gamma)$ be a complete set of representatives for conjugacy classes of $\Gamma$.
For each $z\in\mathcal{C}(\Gamma)$, let $N_{z}$ be the centralizer of $z$ in $\Gamma$, then
$$\theta^{\omega}_{z}:N_{z}^{2}\to U(1), \hspace{10mm} (h,h')\mapsto \theta^{\omega}_{z}(h,h')$$
is a 2-cocycle;
let $\mathcal{R}^{\omega}_{z}$ denote the set of irreducible $\theta^{\omega}_{z}$-projective representations of $N_{z}$.
For each $\rho\in\mathcal{R}^{\omega}_{z}$, let $\chi_{\rho}\in E$ be the unique function with
\begin{align}
\chi_{\rho}(x,h)=\delta_{x,z}\cdot{\rm tr}(\rho(h)) \ \ \ \ \ \text{for\ all\ \ } x\in\mathcal{C}(\Gamma).
\end{align}

By Theorem 23 of \cite{Wi08},  a canonical orthonormal basis for $E$ is given by
\begin{align}
\{\chi_{\rho}\colon \rho\in\Lambda^{\omega}\}, \ \ \ \ \ \text{with\ \ } \Lambda^{\omega}=\bigsqcup_{z\in\mathcal{C}(\Gamma)}\mathcal{R}^{\omega}_{z}.
\end{align}

It is well-known that the mapping class group of $\Sigma_{1}$ is isomorphic to $\textrm{SL}(2,\mathbb{Z})$,
which is generated by
\begin{align}
q=\left(\begin{array}{ll} 0 & 1 \\ -1 & 0 \\ \end{array}\right) \ \ \ \ \ \text{and\ \ \ \ \ }
t=\left(\begin{array}{ll} 1 & 0 \\ 1 & 1 \\ \end{array}\right).
\end{align}
If $f=\left(\begin{array}{cc} a & b \\ c & d \\ \end{array}\right)$, by which we mean that the diffeomorphism $f:\Sigma_{1}\to\Sigma_{1}$ represents $\left(\begin{array}{cc} a & b \\ c & d \\ \end{array}\right)$, then it is induced by the mapping
$\mathbb{R}^2\to\mathbb{R}^2$ which sends $(v,w)$ to $(v,w)\left(\begin{array}{cc} a & b \\ c & d \\ \end{array}\right)$,
hence
\begin{align}
f^{\ast}(\phi_{x,h})=\phi_{x^{a}h^{b},x^{c}h^{d}}, \label{eq:action2}
\end{align}
and it follows from (\ref{eq:diff}) and (\ref{eq:iso}) that for each $\varphi\in E$,
\begin{align}
(f_{\ast}(\varphi))(x,h)=\omega(\phi_{x^{a}h^{b},x^{c}h^{d}}; f_{\#}^{-1}(\xi_{0})-\xi_{0})\cdot\varphi(x^{a}h^{b},x^{c}h^{d}). \label{eq:action3}
\end{align}

Let $P$ be the pair of pants, %viewed as a cobordism from $S^{1}\sqcup S^{1}$ to $S^{1}$,
then $P\times S^{1}$ is a cobordism from $\Sigma_{1}\sqcup\Sigma_{1}$ to $\Sigma_{1}$. The linear map
${\rm mul}=Z^{\omega}(P\times S^{1}): E\otimes E\rightarrow E$ takes the form
\begin{align}
({\rm mul}(\varphi_{1}\otimes\varphi_{2}))(x,h)
=\sum\limits_{x_{1}x_{2}=x}\gamma^{\omega}_{h}(x_{1},x_{2})\varphi_{1}(x_{1},h)\varphi_{2}(x_{2},h).
\end{align}
Furthermore, actually it can be diagonalized:
\begin{align}
\textrm{mul}(\psi_{\rho}\otimes \psi_{\rho'})&=\delta_{\rho,\rho'}\cdot D_{\rho}\psi_{\rho}, \\
\text{with\ \ \ \ } \psi_{\rho}=q_{\ast}(\chi_{\rho}), \ \ \ \  D_{\rho}&=\frac{\# N_{z}}{\chi_{\rho}(z,e)} \ \text{for\ } \rho\in\mathcal{R}^{\omega}_{z}.
\end{align}
By dualizing, ${\rm com}=Z(-P\times S^{1}):E\to E\otimes E$ can be expressed as
\begin{align}
{\rm com}(\psi_{\rho})=D_{\rho}\cdot\psi_{\rho}\otimes\psi_{\rho}.
\end{align}

\subsection{Cycle-cocycle calculus}

Let $\mathbb{Z}_{m}=\mathbb{Z}/m\mathbb{Z}$, regarded as a quotient ring of $\mathbb{Z}$ when necessary. For $a\in\mathbb{Z}$, denote its image under $\mathbb{Z}\to\mathbb{Z}_{m}$ also by $a$.
Let
\begin{align}
\mathbb{Z}_{m}\rightarrow\{0,1,\cdots,m-1\}, \hspace{10mm} x\mapsto\tilde{x}
\end{align}
be the obvious bijection.

By Proposition 2.3 of \cite{HLY12}, a complete set of representatives of elements of
$H^{3}(B\mathbb{Z}_{m};U(1))\cong\mathbb{Z}_{m}$ can be given by
\begin{align}
\mu_{\ell}:\mathbb{Z}_{m}^{3}\rightarrow U(1), \hspace{5mm} (x,y,z)\mapsto\zeta_{m}^{\widetilde{\ell x}\lfloor\frac{\tilde{y}+\tilde{z}}{m}\rfloor},
\hspace{5mm} \ell\in\mathbb{Z}_{m},
\end{align}
where $\zeta_{m}=e^{2\pi\sqrt{-1}/m}$ and $\lfloor\ \rfloor$ is the floor function.
Observe that
\begin{align}
\theta^{\ell}_{h}(x,y):=\theta^{\mu_{\ell}}_{h}(x,y)=\mu_{\ell}(h,x,y)=\zeta_{m}^{\widetilde{\ell h}\lfloor\frac{\tilde{x}+\tilde{y}}{m}\rfloor}
=\zeta_{m^{2}}^{\widetilde{\ell h}(\tilde{x}+\tilde{y}-\widetilde{x+y})}.
\label{eq:theta2}
\end{align}

\begin{nota}
\rm For each $z\in\Gamma$, let $m_{z}$ be its order, and let $\ell_{z}\in\mathbb{Z}_{m_{z}}$ be the image of $[\omega]$ under the pull-back
$H^{3}(B\Gamma;U(1))\stackrel{\iota^{\ast}_{z}}\longrightarrow H^{3}(B\langle z\rangle;U(1))\cong\mathbb{Z}_{m_{z}}$
induced by the inclusion $\iota_{z}:\langle z\rangle\hookrightarrow\Gamma$.
Then the function
\begin{align}
\omega_{z}:\mathbb{Z}_{m_{z}}^{3}\to U(1), \hspace{10mm} (a,b,c)\to \omega(z^{a},z^{b},z^{c})
\end{align}
is cohomologous to $\mu_{\ell_{z}}$; take $\beta_{z}:\mathbb{Z}_{m_{z}}^{2}\to U(1)$ such that
$\omega_{z}=\mu_{\ell_{z}}\cdot\delta\beta_{z}.$
\end{nota}

Put
\begin{align}
\epsilon_{z}(a,b)=\frac{\beta_{z}(a,b)}{\beta_{z}(b,a)}.
\end{align}
Observe that (using the notations (\ref{eq:theta}) and (\ref{eq:theta2}))
\begin{align}
\frac{\theta_{a}^{\omega_{z}}(b,c)}{\theta_{a}^{\ell_{z}}(b,c)}&=\frac{(\delta\beta)(a,b,c)\cdot(\delta\beta)(b,c,a)}{(\delta\beta)(b,a,c)}
\nonumber \\
&=\frac{\beta(b,c)\beta(a,b+c)}{\beta(a+b,c)\beta(a,b)}\cdot\frac{\beta(c,a)\beta(b,c+a)}{\beta(b+c,a)\beta(b,c)}
\cdot\frac{\beta(b+a,c)\beta(b,a)}{\beta(a,c)\beta(b,a+c)}  \nonumber \\
&=\frac{\epsilon_{z}(a,b+c)}{\epsilon_{z}(a,b)\epsilon_{z}(a,c)}.  \label{eq:theta-epsilon}
\end{align}

Denote the residue of $b$ modulo $m_{z}$ by $\tilde{b}$.

When $b\not\equiv 0\pmod{m_{z}}$, applying (\ref{eq:theta-epsilon}) repeatedly, we obtain
\begin{align}
\epsilon_z(a,b)=\epsilon_z(a,\tilde{b})=\epsilon_z(a,1)^{\tilde{b}}\cdot\prod\limits_{j=0}^{\tilde{b}-1}\theta_{a}^{\omega_{z}}(j,1)
=\epsilon_z(a,1)^{\tilde{b}}\cdot\prod\limits_{j=0}^{\tilde{b}-1}\theta_{a}^{\omega_{z}}(j,1),    \label{eq:epsilon}
\end{align}
noting $\theta_{a}^{\ell_{z}}(j,1)=1$ for $0\le j\le\tilde{b}-1<m_{z}-1$; when $b\equiv 0\pmod{m_{z}}$, (\ref{eq:epsilon}) holds trivially.
As a special case of (\ref{eq:epsilon}),
$$\epsilon_z(1,a)=\prod\limits_{j=0}^{\tilde{a}-1}\theta_{1}^{\omega_{z}}(j,1)=\prod\limits_{j=0}^{\tilde{a}-1}\omega_{z}(1,j,1),$$
thus actually
\begin{align}
\epsilon_z(a,b)=\prod\limits_{j=0}^{\tilde{a}-1}\omega_{z}(1,j,1)^{-\tilde{b}}\cdot
\prod\limits_{j=0}^{\tilde{b}-1}\theta_{a}^{\omega_{z}}(j,1).
\end{align}

Put
\begin{align}
\kappa^{\omega}_{a,b}(z)=\zeta_{m_{z}^{2}}^{\widetilde{\ell_{z}}(-b\tilde{a}-a(\widetilde{-b}))}\epsilon_{z}(a,-b).
\end{align}

The following formula is the main achievement of this article:

\begin{lem} \label{lem:key}
If $f=\left(\begin{array}{ll} a & b \\ c & d \\ \end{array}\right)$, then
\begin{align}
\omega(\phi_{e,z};f^{-1}_{\#}(\xi_{0})-\xi_{0})=\kappa^{\omega}_{a,b}(z).
\label{eq:lem}
\end{align}
\end{lem}

\begin{proof}
The result is trivial when $f$ is the identity. We prove the lemma by showing that if it is true for $f$, then it is also true for $fq=\left(\begin{array}{cc} -b & a \\ -d & c \\ \end{array}\right)$,
$fq^{-1}=\left(\begin{array}{cc} b & -a \\ d & -c \\ \end{array}\right)$,
$ft=\left(\begin{array}{cc} a+b & b \\ c+d & d \\ \end{array}\right)$ and
$ft^{-1}q=\left(\begin{array}{cc} -b & a-b \\ -d & c-d \\ \end{array}\right)$.

Suppose (\ref{eq:lem}) holds for $f$.
Abbreviate $\zeta_{m_{z}^{2}}^{\widetilde{\ell_{z}}}$, $\beta_{z}$, $\epsilon_{z}$ to $\zeta, \beta,\epsilon$, respectively.

\begin{figure} [h]
  \centering
  \includegraphics[width=0.6\textwidth]{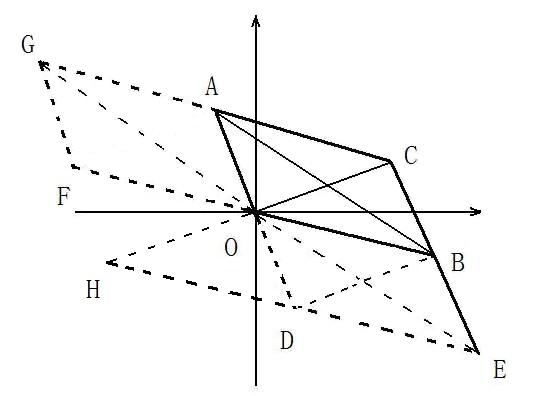}\\
  \caption{Computing the cycle $f^{-1}_{\#}(\xi_{0})-\xi_{0}$} \label{fig:cycle}
\end{figure}

Let $A=(-c,a), B=(d,-b), C=(d-c,a-b)$, and let $D,E,F,G,H$ be the points shown in Figure \ref{fig:cycle}. Then
\begin{align*}
f^{-1}_{\#}(\xi_{0})&=[OAC]-[OBC], \\
(fq)^{-1}_{\#}(\xi_{0})&=[OBE]-[ODE]=[ACB]-[AOB], \\
(fq^{-1})^{-1}_{\#}(\xi_{0})&=[OFG]-[OAG]=[BOA]-[BCA], \\
(ft)^{-1}_{\#}(\xi_{0})&=[OGA]-[OBA]=[BAC]-[OBA], \\
(ft^{-1}q)^{-1}_{\#}(\xi_{0})&=[OBD]-[OHD]=[ACO]-[COB].
\end{align*}
\begin{enumerate}
  \item[\rm(i)] Computing directly, one obtains
      \begin{align*}
      &f^{-1}_{\#}(\xi_{0})-(fq)^{-1}_{\#}(\xi_{0})=([OAC]-[OBC])-([ACB]-[AOB]) \\
      =\ &\partial([OBCB]+[OAOB]-[OACB]+[OOOB]+[OBBB]),
      \end{align*}
      hence
      \begin{align*}
      &\omega(\phi_{e,z};f^{-1}_{\#}(\xi_{0})-(fq)^{-1}_{\#}(\xi_{0})) \\
      =\ &\omega_{z}(-b,a,-a)\omega_{z}(a,-a,-b)\omega_{z}(a,-b,-a)^{-1} =\theta_{-b}^{\omega_{z}}(a,-a);
      \end{align*}
      by (\ref{eq:theta2}), (\ref{eq:theta-epsilon}) and the inductive hypothesis,
      \begin{align*}
      &\omega(\phi_{e,z};(fq)^{-1}_{\#}(\xi_{0})-\xi_{0})  \\
      =\ &\omega(\phi_{e,z};f^{-1}_{\#}(\xi_{0})-\xi_{0})\cdot\omega(\phi_{e,z};f^{-1}_{\#}(\xi_{0})-(fq)^{-1}_{\#}(\xi_{0}))^{-1}  \\
      =\ &\zeta^{-b\tilde{a}-a(\widetilde{-b})}\epsilon(a,-b)
      \cdot\zeta^{b(\tilde{a}+\widetilde{-a})}\epsilon(-b,a)\epsilon(-b,-a)  \\
      =\ &\zeta^{-a(\widetilde{-b})+b(\widetilde{-a})}\epsilon(-b,-a)=\kappa^{\omega}_{-b,a}(z),
      \end{align*}
      thus (\ref{eq:lem}) holds for $fq$.
  \item[\rm(ii)] One may check that
      \begin{align*}
      &(fq^{-1})^{-1}_{\#}(\xi_{0})-f^{-1}_{\#}(\xi_{0})=([BOA]-[BCA])-([OAC]-[OBC]) \\
      =\ &\partial([OACA]+[OBOA]-[OBCA]+[OOOA]+[OAAA]),
      \end{align*}
      then (\ref{eq:lem}) for $fq^{-1}$ is verified through
      \begin{align*}
      &\omega(\phi_{e,z};(fq^{-1})^{-1}_{\#}(\xi_{0})-f^{-1}_{\#}(\xi_{0}))  \\
      =\ &\omega_{z}(a,-b,b)\omega_{z}(-b,b,a)\omega_{z}(-b,a,b)^{-1}  =\theta_{a}^{\omega_{z}}(-b,b)
      \end{align*}
      so that
      \begin{align*}
      &\omega(\phi_{e,z};(fq^{-1})^{-1}_{\#}(\xi_{0})-\xi_{0})  \\
      =\ &\zeta^{-b\tilde{a}-a(\widetilde{-b})}\epsilon(a,-b)
      \cdot\zeta^{a(\widetilde{-b}+\tilde{b})}\epsilon(b,a)\epsilon(-b,a) \\
      =\ &\zeta^{a\tilde{b}-b\tilde{a}}\epsilon(b,a)=\kappa^{\omega}_{b,-a}(z).
      \end{align*}
  \item[\rm(iii)] Using $\partial[OBAC]=([BAC]-[OBA])-([OAC]-[OBC])$, we obtain
      \begin{align*}
      &\omega(\phi_{e,z};(ft)^{-1}_{\#}(\xi_{0})-\xi_{0}) \\
      =\ &\omega(\phi_{e,z};f^{-1}_{\#}(\xi_{0})-\xi_{0})\cdot\omega_{z}(-b,a+b,-b) \\
      =\ &\zeta^{-b\tilde{a}-a(\widetilde{-b})}\epsilon(a,-b)
         \cdot\zeta^{-b(\widetilde{a+b}+\widetilde{-b}-\widetilde{a})}\frac{\beta(a+b,-b)\beta(-b,a)}{\beta(a,-b)\beta(-b,a+b)} \\
      =\ &\zeta^{-b(\widetilde{a+b})-(a+b)(\widetilde{-b})}\epsilon(a+b,-b)=\kappa^{\omega}_{a+b,b}(z),
      \end{align*}
      verifying (\ref{eq:lem}) for $ft$.
  \item[\rm(iv)] Finally, since $[ACO]-[COB]=[ACB]-[AOB]-\partial[ACOB]$, we have
      \begin{align*}
      &\omega(\phi_{e,z};(ft^{-1}q)^{-1}_{\#}(\xi_{0})-\xi_{0})  \\
      =\ &\omega(\phi_{e,z};(fq)^{-1}_{\#}(\xi_{0})-\xi_{0})\cdot\omega_{z}(-b,b-a,-b)^{-1} \\
      =\ &\zeta^{-a(\widetilde{-b})+b(\widetilde{-a})}\epsilon(-b,-a)\cdot\zeta^{b(\widetilde{b-a}+\widetilde{-b}-\tilde{-a})}
          \frac{\beta(-a,-b)\beta(-b,b-a)}{\beta(b-a,-b)\beta(-b,-a)}   \\
      =\ &\zeta^{(b-a)(\widetilde{-b})+b(\widetilde{b-a})}\epsilon(-b,b-a)=\kappa^{\omega}_{-b,a-b}(z),
      \end{align*}
      hence (\ref{eq:lem}) is true for $ft^{-1}q$.
\end{enumerate}

\end{proof}

\begin{rmk}  \label{rmk:n-th}
\rm By the same method, one can show that for each integer $n$,
\begin{align}
\omega(\phi_{e,z^n};f^{-1}_{\#}(\xi_{0})-\xi_{0})=\kappa^{\omega}_{na,nb}(z).
\end{align}
\end{rmk}

\subsection{Formula for Seifert 3-manifolds with orientable bases}

Recall that (see \cite{Topof3mfd}), an orientable \emph{Seifert 3-manifold} can be obtained as follows. Take a
circle bundle $S^{1}\rightarrow X\rightarrow R$ where $X$ is an orientable  manifold and $R$ is a surface with $\partial R=\sqcup_{n}S^{1}$,
so that $\partial X=\sqcup_{n}\Sigma_{1}$; glue $n$ copies of solid torus ${\rm ST}:=D\times S^{1}$ onto $X$ along the boundary
tori via diffeomorphisms $f_{j}:\Sigma_{1}\rightarrow\Sigma_{1},j=1,\cdots,n$. The resulting manifold $M$ is a {\it Seifert 3-manifold}, and
the closed surface $R\cup(\sqcup_{n}D)$ is called the \emph{base}.
%and the images of $S^{1}\times 0\subset\textrm{ST}$ are called the \emph{exceptional fibers}.

When $f_{j}=\left(\begin{array}{cc} a_{j} & b_{j} \\ c_{j} & d_{j} \end{array}\right)$ and $R$ is orientable of genus $g$, denote $M$ as
\begin{align}
M_{O}(g;(a_{1},b_{1}),\cdots,(a_{n},b_{n})).
\end{align}

\begin{figure} [h]
  \centering
  \includegraphics[width=0.4\textwidth]{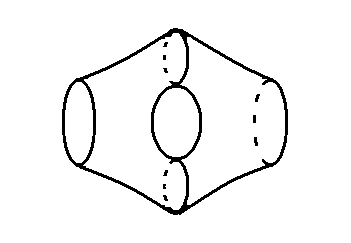}\\
  \caption{$\Sigma_{1;1,1}=-P\cup P$} \label{fig:P-P}
\end{figure}

Let $\Sigma_{g;n,1}$ denote the surface with genus $g$ and $n+1$ small disks removed, regarded as a cobordism from $\sqcup_{n}S^{1}$ to $S^{1}$.
As illustrated in Figure \ref{fig:P-P}, $\Sigma_{1;1,1}$ is the composite of $-P$ and $P$, so
\begin{align}
Z^{\omega}(\Sigma_{1;1,1}\times S^{1})={\rm mul}\circ {\rm com}:E\to E, \hspace{10mm} \psi_{\rho}\mapsto D_{\rho}^{2}\psi_{\rho}.
\end{align}
%Similarly as in \cite{Ch12}, applying the gluing axiom we see that the linear map $Z^{\omega}(\Sigma_{g;n,1}\times S^{1}):E^{\otimes n}\to E$ can be expressed as
%\begin{align*}
%\psi_{\rho_{1}}\otimes\cdots\otimes \psi_{\rho_{n}}\mapsto \delta_{\rho_{1},\cdots,\rho_{n}}\cdot D_{\rho_{1}}^{n+2g-1}\psi_{\rho_{1}}.
%\end{align*}

Since $Z^{\omega}({\rm ST})(x,h)=\delta_{x,e}$, we have
\begin{align}
Z^\omega({\rm ST})=
\sum\limits_{\rho\in\Lambda^\omega}(q^{-1}_{\ast}(Z^\omega({\rm ST})),\chi_{\rho})\cdot\psi_{\rho}= \sum\limits_{\rho\in\Lambda^\omega}\frac{1}{D_{\rho}}\psi_{\rho};
\end{align}
as a morphism $E\rightarrow \mathbb{C}$,
\begin{align}
Z^\omega(-{\rm ST})(\psi_{\rho})=(\psi_{\rho},Z^\omega({\rm ST}))=\frac{1}{D_{\rho}}.
\end{align}

For $f=\left(\begin{array}{ll} a & b \\ c & d \\ \end{array}\right)$, by (\ref{eq:action3}) and (\ref{eq:lem}), we have
\begin{align*}
(f_{\ast}(Z^{\omega}({\rm ST})))(x,h)
=\omega(\phi_{e,x^{c}h^{d}};f_{\#}^{-1}(\xi_{0})-\xi_{0})\cdot\delta_{x^{a}h^{b},e}
=\kappa^{\omega}_{a,b}(x^{c}h^{d})\cdot\delta_{x^{a}h^{b},e},
\end{align*}
hence
\begin{align*}
f_{\ast}(Z^{\omega}({\rm ST}))=\frac{1}{\#\Gamma}\cdot\sum\limits_{\rho\in\Lambda^{\omega}}
\left(\sum\limits_{xh=hx\atop x^{a}h^{b}=e}\kappa^{\omega}_{a,b}(x^{c}h^{d})\overline{\psi_{\rho}(x,h)}\right)\cdot\psi_{\rho}.
\end{align*}

\begin{figure} [h]
  \centering
  \includegraphics[width=0.4\textwidth]{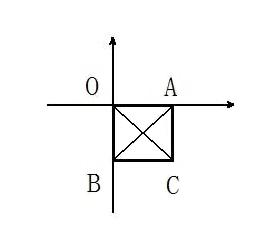}\\
  \caption{Computing $q^{-1}_{\#}(\xi_{0})-\xi_{0}$}  \label{fig:cycle2}
\end{figure}

Since, as Figure \ref{fig:cycle2} shows,
\begin{align*}
&q_{\#}^{-1}(\xi_{0})-\xi_{0}=([OAC]-[OBC])-([BOA]-[BCA]) \\
=\ &\partial([OBCA]-[OACA]-[OBOA]+[OOOA]+[OAAA]),
\end{align*}
we have
$$\omega(\phi_{h,x^{-1}};q^{-1}_{\#}(\xi_{0})-\xi_{0})
=\frac{\omega(x,h,x^{-1})}{\omega(h,x,x^{-1})\omega(x,x^{-1},h)}=\theta^{\omega}_{h}(x,x^{-1})^{-1},$$
hence
\begin{align}
\psi_{\rho}(x,h)=(q_{\ast}(\chi_{\rho}))(x,h)=\chi_{\rho}(h,x^{-1})\theta^{\omega}_{h}(x,x^{-1})^{-1}=\overline{\chi_{\rho}(h,x)},
\end{align}
using $\rho(x)\rho(x^{-1})=\theta^{\omega}_{h}(x,x^{-1})\rho(e)$.
Consequently,
\begin{align*}
\sum\limits_{xh=hx\atop x^{a}h^{b}=e}\kappa^{\omega}_{a,b}(x^{c}h^{d})\overline{\psi_{\rho}(x,h)}=
\sum\limits_{xh=hx\atop x^{a}h^{b}=e}\kappa^{\omega}_{a,b}(x^{c}h^{d})\chi_{\rho}(h,x)= \sum\limits_{z\in\Gamma}\kappa^{\omega}_{a,b}(z)\chi_{\rho}(z^{a},z^{-b}),
\end{align*}
where the last equality is obtained by an argument similarly as in the proof of Lemma 3.1 of \cite{Ch12};
denote
\begin{align}
\eta_{\rho}^{\omega}(a,b)=\sum\limits_{z\in\Gamma}\kappa^{\omega}_{a,b}(z)\chi_{\rho}(z^{a},z^{-b}),  \label{eq:eta}
\end{align}
then for $j=1,\ldots,n$,
\begin{align}
(f_{j})_{\ast}(Z^{\omega}({\rm ST}))
=\frac{1}{\#\Gamma}\cdot\sum\limits_{\rho\in\Lambda^{\omega}}\eta^{\omega}_{\rho}(a_{j},b_{j})\cdot\psi_{\rho}.
\end{align}

\begin{thm} \label{thm:main}
We have the following formula:
\begin{align*}
Z^{\omega}(M_{O}(g;(a_{1},b_{1}),\cdots,(a_{n},b_{n})))
=\frac{1}{(\#\Gamma)^{n}}\cdot\sum\limits_{\rho\in\Lambda^{\omega}}D_{\rho}^{n+2g-2}\prod\limits_{j=1}^{n}\eta^{\omega}_{\rho}(a_{j},b_{j}).
\end{align*}
\end{thm}

\begin{proof}
%Modifying the proof of Theorem 3.3 of \cite{Ch12}, one can establish
Since $\Sigma_{g;1,1}$ can be obtained by gluing $g$ $\Sigma_{1;1,1}$'s successively, we have
\begin{align}
Z^\omega(\Sigma_{g;1,1}\times S^{1})=({\rm mul}\circ {\rm com})^{g}:E\rightarrow E,\qquad \psi_{\rho}\mapsto(D_{\rho})^{2g}\psi_{\rho}.
\end{align}
For $p>0$, $\Sigma_{0;p,1}$ can be obtained by gluing $(p-1)$ $P$'s successively, hence
$Z^\omega(\Sigma_{0;p,1}\times S^{1}):E^{\otimes p}\rightarrow E$ is equal to
\begin{align*}
{\rm mul} &\circ({\rm id}\otimes {\rm mul})\circ\cdots\circ({\rm id}\otimes\cdots\otimes {\rm id}\otimes {\rm mul}),\nonumber\\
&\psi_{\rho_{1}}\otimes\cdots\otimes\psi_{\rho_{p}}\mapsto(D_{\rho_{1}})^{p-1}\delta_{\rho_{1},\cdots,\rho_{p}}\cdot\psi_{\rho_{1}}.
\end{align*}
Dually, for $q>0$, $Z^\omega(\Sigma_{0;1,q}\times S^{1}):E\rightarrow E^{\otimes q}$ is equal to
\begin{align*}
({\rm id}\otimes\cdots\otimes {\rm id}\otimes {\rm com})\circ\cdots\otimes({\rm id}\otimes {\rm com})\circ {\rm com}, \qquad
\psi_{\rho}\mapsto(D_{\rho})^{q-1}\psi_{\rho}^{\otimes q}.
\end{align*}
So $Z^\omega(\Sigma_{g;p,q}\times S^{1}):E^{\otimes p}\rightarrow E^{\otimes q}$ is equal to the composite
\begin{align*}
E^{\otimes p}\xrightarrow[]{Z^\omega(\Sigma_{0;p,1}\times S^{1})}E&\xrightarrow[]{Z^\omega(\Sigma_{g;1,1}\times S^{1})}E\xrightarrow[]
{Z^\omega(\Sigma_{0;1,q}\times S^{1})}E^{\otimes q}, \\
\psi_{\rho_{1}}\otimes \cdots\otimes\psi_{\rho_{p}}&\mapsto (D_{\rho_{1}})^{p+q+2g-2}
\delta_{\rho_{1},\cdots,\rho_{p}}\cdot\psi_{\rho_{1}}^{\otimes q}.
\end{align*}

Let $M'$ be $M_{O}(g;(a_{1},b_{1}),\ldots,(a_{n},b_{n}))$
with an ${\rm ST}$ deleted. It is obtained by gluing $n$ ${\rm ST}$'s onto $\Sigma_{g;n,1}\times S^{1}$, using
$f_{1},\ldots,f_n$.
Then $Z^\omega(M')$ is the composite
\begin{align}
\mathbb{C}\cong &\mathbb{C}^{\otimes n}\xrightarrow[]{\bigotimes\limits_{j=1}^{n}((f_{j})_{\ast}Z^\omega({\rm ST}))}E^{\otimes n}\xrightarrow[]{Z^\omega(\Sigma_{g;n,1}\times S^{1})}E. \nonumber \\
1&\mapsto\sum\limits_{\rho\in\Lambda^\omega}\left(\frac{(D_{\rho})^{n+2g-1}}{(\#\Gamma)^{n}}\cdot\prod\limits_{j=1}^{n}
\eta^\omega_{\rho}(a_{j},b_{j})\right)\psi_{\rho}. \label{eq:Z(M')}
\end{align}

The theorem follows from that $Z^\omega(M_{O}(g;(a_{1},b_{1}),\cdots,(a_{n},b_{n})))$ is the composite
$\mathbb{C}\xrightarrow[]{Z^\omega(M')} E\xrightarrow[]{Z^\omega(-{\rm ST})}\mathbb{C}.$
%\end{align}
\end{proof}

\begin{exmp}
\rm
Let $\Gamma=\mathbb{Z}_{m}$ and $\omega=\mu_{\ell}$.

For each $h\in\mathcal{C}(\mathbb{Z}_{m})=\mathbb{Z}_m$, we have $N_{h}=\mathbb{Z}_{m}$.
It is easy to see that there is a bijective correspondence between $\theta^{\ell}_{h}$-projective representations $\rho$ and ordinary representations $\tilde{\rho}$, through
$$\tilde{\rho}(x)=\zeta_{m^2}^{-\ell\tilde{h}\tilde{x}}\cdot\rho(x),$$
hence each irreducible $\theta^{\ell}_{h}$-projective representation of $\mathbb{Z}_{m}$ is given by
\begin{align}
\rho^{\ell}_{h,s}(x)=\zeta_{m^2}^{\ell\tilde{h}\tilde{x}+m(\widetilde{sx})}
\end{align}
for a unique $s\in\mathbb{Z}_{m}$.
Computing directly or referring to \cite{ACM04} Proposition 8,
\begin{align}
\chi_{\rho^{\ell}_{h,s}}(x,y)=\delta_{h,x}\cdot\zeta_{m^{2}}^{\tilde{\ell}\tilde{h}\tilde{y}+m\widetilde{sy}}.   \label{eq:chi-cyclic}
\end{align}

By Remark \ref{rmk:n-th}, for each $u\in\mathbb{Z}_m$,
$$\kappa^{\mu_{\ell}}(a,b)(u)=\kappa^{\mu_\ell}(\tilde{u}a,\tilde{u}b)(1)=
\zeta_{m^2}^{\tilde{\ell}(-b\tilde{u}(\widetilde{au})-a\tilde{u}(\widetilde{-bu}))}.$$
By careful computation, using (\ref{eq:eta}) and (\ref{eq:chi-cyclic}),
\begin{align}
\eta_{h,s}^{\mu_{\ell}}(a,b)=\sum\limits_{u:au=h}\zeta_{m^{2}}^{\tilde{\ell}ab\tilde{u}^{2}-(2\tilde{\ell}\tilde{h}+m\tilde{s})b\tilde{u}}.
\end{align}

For $j\in\{1,\cdots,n\}$, let
$d_{j}=(a_{j},m)$, $a'_{j}=a_{j}/d_{j}$, $m_{j}=m/d_{j}$.
Let $c_{j}$ be the unique integer with
$0\leqslant c_{j}<m_{j}$ and $a'_{j}c_{j}\equiv 1\pmod{m_{j}}$; suppose $a'_{j}c_{j}=1+t_{j}m_{j}$.
Let $d=[d_{1},\cdots,d_{n}]$, and let
$\tilde{h}'=\tilde{h}/d$, $\tilde{h}_{j}=\tilde{h}/d_{j}=\tilde{h}'d/d_{j}$.

We have
\begin{align*}
\eta_{h,s}^{\mu_{\ell}}(a_{j},b_{j})&=\ \sum\limits_{k=0}^{d_{j}-1}
\zeta_{m^{2}}^{b_{j}[\tilde{\ell}a_{j}(km_{j}+c_{j}\tilde{h}_{j})^{2}-(2\tilde{\ell}\tilde{h}_{j}d_{j}+m\tilde{s})(km_{j}+c_{j}\tilde{h}_{j})]} \nonumber \\
&=\ \zeta_{mm_{j}}^{-\tilde{\ell}c_{j}\tilde{h}_{j}^{2}}\cdot\zeta_{m}^{(\tilde{\ell}t_{j}-\tilde{s})c_{j}\tilde{h}_{j}}\cdot\sum\limits_{k=0}^{d_{j}-1}
\zeta_{d_{j}}^{\tilde{\ell}a'_{j}k^{2}+(2\tilde{h}_{j}t_{j}-\tilde{s})k}. \label{eq:sum}
\end{align*}

For a prime number $p$, let $(a/p)$ denote the Legendre symbol (see Page 51 of \cite{GTM84}), and let
$g_{1}(\chi)=\sum\limits_{k=1}^{p-1}\chi(k)\zeta_{p}^{k}$ be the Gauss sum (see Page 91 of \cite{GTM84}) with $\chi$
the unique character of order 2.
Put
\begin{align*}
S_{p}(a)=\sum\limits_{k=0}^{p-1}\zeta_{p}^{ak^{2}}=
\left\{\begin{array}{ll}  \frac{1}{2}(1+(a/p))g_{1}(\chi), & p\nmid a, \\  p, &p\mid a. \end{array}\right.
\end{align*}

The $\eta_{h,s}^{\mu_{\ell}}(a_j,b_j)$'s, and also $Z^{\mu_\ell}(M_{O}(g;(a_{1},b_{1}),\cdots,(a_{n},b_{n})))$, can be evaluated using $S_{p}(a)$ for $p\mid m$. In general the expression is very complicated.
\end{exmp}

\end{document}